\documentclass[11pt,a4paper,twoside]{article}

\usepackage{dsfont}
\usepackage{amssymb,amscd,amsfonts,amsbsy,amsmath,amsthm}
\usepackage{enumerate}
\usepackage{mathrsfs}
\usepackage{epsf,epsfig,esint}
\usepackage{pdfsync}
\usepackage[all]{xy}
\usepackage{pdfpages}

\numberwithin{equation}{section}

\newtheorem{proposition}{Proposition}[section]
\newtheorem{theorem}[proposition]{Theorem}

\newtheorem{corollary}[proposition]{Corollary}

\theoremstyle{definition}

\newtheorem{remark}[proposition]{Remark}

\usepackage{color}

\title{Behaviour of the Stokes operators under domain perturbation
}

\author{Sylvie Monniaux\,
\thanks{Aix Marseille Universit\'e, CNRS, Centrale Marseille, I2M UMR 7373, 13453, 
Marseille, France - {\tt sylvie.monniaux@univ-amu.fr}}
}

\date{ }

\begin{document}

\maketitle

\begin{abstract}  
Depending of the geometry of the domain, one can define --at least-- three
different Stokes operators with Dirichlet boundary conditions. We describe how
the resolvents of these Stokes operators converge with respect to a converging
sequence of domains.
\end{abstract}


\section{Introduction}
\label{sec:intro}

Let $\Omega$ denote an open connected subset of ${\mathds{R}}^d$. We do not
impose any regularity of the boundary $\partial\Omega$ of the domain $\Omega$ and 
possibly $\Omega$ is unbounded. To avoid too many cases, we will however assume
that the $d$-dimensional Hausdorff measure of $\partial\Omega$ is zero.

We denote by ${\mathcal{D}}={\mathscr{C}}_c^\infty(\Omega,{\mathds{R}}^d)$ the 
space of smooth vector fields with compact support in $\Omega$. Let ${\mathcal{D}}'$
denote its dual, the space of (vector valued) distributions on $\Omega$.

\subsection*{Acknowledgements}
The author aknowledges the partial support by the {\sc anr} project 
{\sc infamie anr-15-ce40-001}. The understanding of this subject has 
benefited from discussions with A.F.M. ter~Elst. The author would also like
to thank the anonymous referee whose remarks greatly improved this manuscript.

\section{Setting}
\label{sec:setting}

\subsection{The Leray orthogonal decomposition of $L^2$}
\label{subsec:Leraydec}

We start with a very important and profound result due to de Rham 
\cite[Chapter IV \S22, Theorem 17']{dR84};
see also \cite[Chapter I \S1.4, Proposition~1.1]{Tem79}.

\begin{theorem}[de Rham]
\label{thm:deRham}
Let $T\in{\mathcal{D}}'$ be a distribution. Then the following two properties are 
equivalent.
\begin{enumerate}[(i)]
\item
${}_{{\mathcal{D}}'}\langle T,\varphi\rangle_{{\mathcal{D}}}=T(\varphi)=0$ for all 
$\varphi\in{\mathcal{D}}$ with ${\rm div}\,\varphi=0$.
\item
There exists a scalar distribution $S\in{\mathscr{C}}_c^\infty(\Omega)'$ such that
$T=\nabla S$ in ${\mathcal{D}}'$.
\end{enumerate}
\end{theorem}

De Rham's theorem has the following corollary.

\begin{corollary}
\label{cor:deRham}
Let $T\in H^{-1}(\Omega,{\mathds{R}}^d)$. Then the following are equivalent.
\begin{enumerate}[{\rm (i)}]
\item
\label{i}
${}_{H^{-1}(\Omega,{\mathds{R}}^d)}\langle T,
\varphi\rangle_{H^1_0(\Omega,{\mathds{R}}^d)}=0$ 
for all $\varphi\in{\mathcal{D}}$ with ${\rm div}\,\varphi=0$.
\item
\label{ii}
There exists a scalar distribution $\pi\in L^2_{\rm loc}(\Omega)$ such that 
$T=\nabla \pi$ in ${\mathcal{D}}'$.
\end{enumerate}
\end{corollary}

\begin{proof}
We only have to show \eqref{i}$\implies$\eqref{ii}. By Theorem~\ref{thm:deRham} 
there exists $S\in{\mathscr{C}}_c^\infty(\Omega)'$ such that $T=\nabla S$. Then 
$\nabla S\in H^{-1}(\Omega,{\mathds{R}}^d)$.
Consequently, $S\in L^2_{\rm loc}(\Omega)$ by \cite[Proposition~1.2]{Tem79} 
(for a direct proof, see also \cite[Lemma~2.2.1]{So01}.
\end{proof}

Denote by ${\mathcal{H}}=L^2(\Omega,{\mathds{R}}^d)$ the square integrable 
vector fields on $\Omega$. 
We endow the vector-valued space ${\mathcal{H}}$ with the scalar product
\[
\langle u,v\rangle_{\mathcal{H}}:=\int_\Omega u\cdot v=
\sum_{j=1}^d\int_\Omega u_jv_j, \quad u,v \in {\mathcal{H}}.
\]
Then ${\mathcal{H}}$ is a Hilbert space. We define the subspace ${\mathscr{G}}$ of 
${\mathcal{H}}$ consisting of gradients by
\begin{equation}
\label{def:G}
{\mathscr{G}}:=\bigl\{\nabla\pi; \pi\in L^2_{\rm loc}(\Omega),
 \nabla\pi\in{\mathcal{H}}\bigr\}.
\end{equation}
As a consequence of Corollary~\ref{cor:deRham}, ${\mathscr{G}}$ is a closed 
subspace of ${\mathcal{H}}$. We denote by ${\mathscr{H}}$ the orthogonal 
subspace of ${\mathscr{G}}$ in ${\mathcal{H}}$, that is
\begin{equation}
\label{def:H}
{\mathscr{H}}=\bigl\{u\in {\mathcal{H}};\langle u,g\rangle_{\mathcal{H}}=0\mbox{ for all }
g\in{\mathscr{G}}\bigr\}.
\end{equation}
Obviously, ${\mathscr{H}}$ is a Hilbert space and one has the orthogonal 
decomposition
\begin{equation}
\label{def:decompH}
{\mathcal{H}}={\mathscr{H}}\stackrel{\bot}{\oplus}{\mathscr{G}}.
\end{equation}
The orthogonal projection from ${\mathcal{H}}$ to ${\mathscr{H}}$ denoted by 
${\mathbb{P}}$ is called the {\tt Leray projection}. It is the adjoint of the canonical 
embedding $J:{\mathscr{H}}\hookrightarrow{\mathcal{H}}$; it verifies 
${\mathbb{P}}Ju=u$ for all $u\in{\mathscr{H}}$.
Next, define the subspace
\begin{equation}
\label{def:D}
{\mathscr{D}}=\bigl\{u\in{\mathcal{D}} ; {\rm div}\,u=0 \mbox{ in }\Omega\bigr\}.
\end{equation}
Then ${\mathscr{D}}\subset {\mathscr{H}}$ and by De Rham's theorem, 
${\mathscr{D}}^\bot={\mathscr{G}}$, so that ${\mathscr{D}}$ is dense in ${\mathscr{H}}$ 
with respect to the $L^2$-norm of ${\mathcal{H}}$.

The canonical embedding $J_0:{\mathscr{D}}\hookrightarrow{\mathcal{D}}$ is the 
restriction of $J$ to ${\mathscr{D}}$. Its adjoint 
$J_0'={\mathbb{P}}_1:{\mathcal{D}}'\to {\mathscr{D}}'$ is therefore an extension of
the Leray projection ${\mathbb{P}}$. A reformulation of de Rham's theorem 
(Thm~\ref{thm:deRham}) is
\[
\ker {\mathbb{P}}_1=\bigl\{T\in{\mathcal{D}}'; {\mathbb{P}}_1T=0\bigr\}
=\bigl\{\nabla S; S\in{\mathscr{C}}_c^\infty(\Omega)'\bigr\}.
\]

\subsection{Another orthogonal decomposition of $L^2$}
\label{subsec:anotherdec}

Since we made the assumption that the $d$-dimensional Hausdorff measure 
of $\partial\Omega$ is zero, we can
identify ${\mathcal{H}}=L^2(\Omega,{\mathds{R}}^d)$ with 
$\bigl\{U_{|_\Omega}; U\in L^2({\mathds{R}}^d,{\mathds{R}}^d), U=0 \mbox{ a.e. in }
{}^c\overline{\Omega}\bigr\}$ 
and define the space ${\mathscr{E}}$ to be the closure in $L^2(\Omega,{\mathds{R}}^d)$ of 
\begin{equation}
\label{def:W}
{\mathscr{W}}:=\bigl\{U_{|_{\Omega}} ; U\in H^1({\mathds{R}}^d,{\mathds{R}}^d), 
U=0\mbox{ a.e. in }{}^c\overline{\Omega} \mbox{ and }{\rm div}\,U=0\mbox{ in }
{\mathds{R}}^d\bigr\}.
\end{equation}
The space ${\mathscr{E}}$ is closed in ${\mathcal{H}}$ by definition and contains ${\mathscr{D}}$, 
and then ${\mathscr{H}}$.
The following decomposition of ${\mathcal{H}}$ holds
\begin{equation}
\label{def:anotherdecompH}
{\mathcal{H}}={\mathscr{E}}\stackrel{\bot}{\oplus}{\mathscr{F}},
\end{equation}
where ${\mathscr{F}}={\mathscr{E}}^\bot$.
Since ${\mathscr{H}}\subset{\mathscr{E}}$, it is obvious that ${\mathscr{F}}
\subset{\mathscr{G}}$.
It is also obvious that $\bigl\{\nabla q_{|_\Omega} ; q\in\dot H^1({\mathds{R}}^d)\bigr\} 
\subset {\mathscr{F}}$: let $u=U_{|_\Omega}\in {\mathscr{W}}$ and 
$q\in\dot H^1({\mathds{R}}^d)$; then
\[
\langle u,\nabla q\rangle_{\mathcal{H}}
=\langle U,\nabla q\rangle_{L^2({\mathds{R}}^d,{\mathds{R}}^d)}=0.
\]
For further use, we will denote by $L:{\mathscr{E}}\hookrightarrow{\mathcal{H}}$
the canonical embedding; its adjoint $L'={\mathbb{Q}}:{\mathcal{H}}\to{\mathscr{E}}$
is the orthogonal projection from ${\mathcal{H}}$ to ${\mathscr{E}}$. The operators
$L$ and ${\mathbb{Q}}$ verify ${\mathbb{Q}} Lu=u$ for all $u\in{\mathscr{E}}$, 
as do $J$ and ${\mathbb{P}}$ in the above setting.

\begin{remark}
When $\Omega\subset {\mathds{R}}^d$ is bounded and smooth enough, 
say with Lipschitz boundary, the spaces ${\mathscr{H}}$ and ${\mathscr{E}}$ 
coincide: they are equal to
\[
{\mathbb{L}}^2_\sigma(\Omega):=
\bigl\{u\in L^2(\Omega,{\mathds{R}}^d); {\rm div}\,u=0\mbox{ in }\Omega\mbox{ and }
\nu\cdot u=0\mbox{ on }\partial\Omega\bigr\},
\]
where ${\rm div}\,u$ is to be taken in the sense of distributions and $\nu(x)$ denotes 
the exterior normal unit vector at $x\in\partial\Omega$, defined
for almost every $x$ in the case of a Lipschitz boundary $\partial\Omega$. Here,
$\nu\cdot u\in H^{-1/2}(\partial\Omega)$ is defined via the integration by parts formula
\[
{}_{H^{-1/2}}\langle \nu\cdot u,\varphi\rangle_{H^{1/2}}=\int_\Omega u\cdot \nabla\Phi
+\int_\Omega {\rm div}\,u\cdot \Phi
\] 
for all $\varphi\in H^{1/2}(\partial\Omega)$ and $\Phi\in H^1(\Omega)$ satisfying
${\rm Tr}_{|_{\partial\Omega}}\Phi=\varphi$.
\end{remark}

The fact that ${\mathscr{H}}={\mathbb{L}}^2_\sigma(\Omega)$ in the case of a 
bounded domain with Lipschitz boundary was proved in \cite[Thm~1.4]{Tem79}.
If $\Omega\subset{\mathds{R}}^d$ has a continuous boundary as in 
\cite[Prop.~2.2]{AD08} (see also \cite{St75}), 
${\mathscr{W}}=\{u\in H^1_0(\Omega,{\mathds{R}}^d); {\rm div}\,u=0\}$.
According to \cite[Thm~1.6]{Tem79}, this latter space is the closure of
${\mathscr{D}}$ in $H^1(\Omega,{\mathds{R}}^d)$ if the boundary of $\Omega$ 
is Lipschitz, so that ${\mathscr{E}}={\mathbb{L}}^2_\sigma(\Omega)={\mathscr{H}}$.

\section{Spaces of divergence-free vector fields}
\label{sec:div-freeVF}

In this section, we introduce several spaces which yield different suitable definitions
of the Stokes operator with Dirichlet boundary conditions.

We start with 
\[
{\mathcal{V}}=H^1_0(\Omega,{\mathds{R}}^d).  
\]
Then ${\mathcal{V}}$ is the closure of ${\mathcal{D}}$ in 
$H^1(\Omega,{\mathds{R}}^d)$. We provide ${\mathcal{V}}$ with the norm 
induced from $H^1(\Omega,{\mathds{R}}^d)$.
Next, we define the space
\[
{\mathcal{W}}:=\bigl\{U_{|_{\Omega}} : U\in H^1({\mathds{R}}^d,{\mathds{R}}^d)
\text{ and }U=0\text{ a.e.\ in }\overline{\Omega}^c\bigr\}.  
\]
Then ${\mathcal{W}}$ is a closed subspace of $H^1(\Omega,{\mathds{R}}^d)$
and we provide ${\mathcal{W}}$ with the norm induced from 
$H^1(\Omega,{\mathds{R}}^d)$.

It is clear that
${\mathcal{D}} \subset {\mathcal{V}}\subseteq {\mathcal{W}} \subset {\mathcal{H}}$.
If $\Omega$ has a continuous boundary, then ${\mathcal{V}}={\mathcal{W}}$ (see 
\cite[pages 24-26]{St75}), but in general ${\mathcal{V}} \neq {\mathcal{W}}$, 
as shown in \cite[Section~7]{AD08}).
Identifying ${\mathcal{H}}$ with its dual, we obtain the Gelfand triples
${\mathcal{V}} \hookrightarrow {\mathcal{H}} \hookrightarrow {\mathcal{V}}'$ and 
${\mathcal{W}} \hookrightarrow {\mathcal{H}} \hookrightarrow {\mathcal{W}}'$.

Let ${\mathscr{V}}$ be the closure of ${\mathscr{D}}$
in ${\mathcal{V}}=H^1_0(\Omega,{\mathds{R}}^d)$ and 
let ${\mathscr{X}}:={\mathcal{V}}\cap{\mathscr{H}}$.
It is straightforward that 
${\mathscr{V}}\subseteq {\mathscr{X}}\subseteq {\mathscr{W}}$.
If $\Omega$ is bounded with Lipschitz boundary, then
${\mathscr{V}} = {\mathscr{X}}= {\mathscr{W}}$
(see \cite[Section~3]{Hey76} and \cite[Theorem~2.2]{LaSo76}), but not in general. 
The famous example for which ${\mathscr{V}}$ is different from 
${\mathscr{X}}$ is the unbounded smooth aperture domain 
(see \cite[Theorem~17]{Hey76}).
The three spaces ${\mathscr{V}}$, ${\mathscr{X}}$ and ${\mathscr{W}}$ 
all contain ${\mathscr{D}}$,   ${\mathscr{V}}$ and ${\mathscr{X}}$ are dense 
subspaces of ${\mathscr{H}}$, ${\mathscr{W}}$ is a dense subspace of
${\mathscr{E}}$ by definition. 
Moreover, ${\mathscr{V}}$ and ${\mathscr{X}}$ are closed in ${\mathcal{V}}$ 
and ${\mathscr{W}}$ is closed in ${\mathcal{W}}$.

\subsection{Weak- and pseudo-Dirichlet Laplacians}
\label{subsec:weakDelta}

We now briefly describe how to define the Laplacian with homogeneous Dirichlet 
boun\-da\-ry conditions in a weak sense: depending on how the boundary conditions 
are modelled, different operators appear. Recall that since we do not impose any 
regularity on the boundary of our domain $\Omega$, it does not make 
sense to talk about traces.
We start by defining the bilinear form 
${\frak{a}} : {\mathcal{W}}\times {\mathcal{W}}\to {\mathds{R}}$ by
\begin{equation}
\label{eq:form-a}
{\frak{a}}(u,v):=\langle \nabla u,\nabla v\rangle_{\mathcal{H}}
= \sum_{j=1}^d \langle \partial_ju,\partial_jv\rangle_{\mathcal{H}}.
\end{equation}
The forms ${\frak{a}}$ and ${\frak{a}}|_{{\mathcal{V}} \times {\mathcal{V}}}$ 
are associated with analytic semigroups of contractions on 
${\mathcal{H}}$ (see, e.g., \cite[\S VI.2]{Kat80}). 
Let $-\Delta_D^\Omega$ be the operator associated with the form
${\frak{a}}|_{{\mathcal{V}} \times {\mathcal{V}}}$ and let 
$-\Delta_D^{\overline{\Omega}}$
be the operator associated with the form ${\frak{a}}$.
Following \cite{AD08}, we call $\Delta_D^\Omega$ the {\tt (weak-)Dirichlet Laplacian}
and $\Delta_D^{\overline{\Omega}}$ the {\tt pseudo-Dirichlet Laplacian}.
They are self-adjoint (unbounded) operators in ${\mathcal{H}}$.
The interest of considering the weak-Dirichlet and the pseudo-Dirichlet 
Laplacians lies in particular in domain perturbation problems.

Let $\Omega,\Omega_1,\Omega_2,\ldots$ be bounded open subsets of 
${\mathds{R}}^d$.
We say that $\Omega_n\uparrow\Omega$ as $n\to\infty$ if 
$\Omega_n\subset\Omega_{n+1}$ for all $n\in{\mathds{N}}$ and 
for each compact subset $K\subset \Omega$ there exists an
$n \in{\mathds{N}}$ with $K\subset\Omega_n$.
We say that $\Omega_n\downarrow\Omega$ as $n\to\infty$ if 
$\Omega_n\supset\Omega_{n+1} \supset\overline{\Omega}$ for all $n\in{\mathds{N}}$
and $\lim_{n \to \infty} |(\Omega_n\cap B)\setminus \overline{\Omega}| = 0$
for every ball $B$, where $| \cdot |$ denotes the Lebesgue measure in 
${\mathds{R}}^d$.

If $f \in {\mathcal{H}}$, then we denote by 
$\tilde f \in L^2({\mathds{R}}^d,{\mathds{R}}^d)$ 
the extension by $0$ of $f$ to ${\mathds{R}}^d$.

The following results have been established in \cite[\S3]{AD08}.
See also \cite[\S6]{Ar01} and \cite[\S6 and \S7]{Da03}.

\begin{proposition}
\label{prop:approxDelta}
Let $\Omega,\Omega_1,\Omega_2,\ldots$ be bounded open subsets of 
${\mathds{R}}^d$.
\begin{enumerate}[a.]
\item 
\label{prop:approxDelta-1}
Suppose that $\Omega_n\uparrow\Omega$ as $n\to\infty$. Then 
\begin{align*}
\lim_{n \to \infty} 
\bigl[\bigl({\rm I}+(-\Delta_D^{\Omega_n})\bigr)^{-1}(f_{|_{\Omega_n}})
\bigr]_{|_\Omega}^{\widetilde{\ \hphantom{|_\Omega}}}
& =  \bigl({\rm I}+(-\Delta_D^\Omega)\bigr)^{-1}f \quad \mbox{and}  \\
\lim_{n \to \infty} 
\bigl[\bigl({\rm I}+(-\Delta_D^{\overline{\Omega}_n})\bigr)^{-1}(f_{|_{\Omega_n}})
\bigr]_{|_\Omega}^{\widetilde{\ \hphantom{|_\Omega}}}
& =  \bigl({\rm I}+(-\Delta_D^\Omega)\bigr)^{-1}f
\end{align*}
in ${\mathcal{H}}$ for all $f\in {\mathcal{H}}$.
\item 
\label{prop:approxDelta-2}
Suppose that $\Omega_n \downarrow\Omega$ as $n\to\infty$.
Then 
\begin{align*}
\lim_{n \to \infty} 
\bigl[\bigl({\rm I}+(-\Delta_D^{\overline{\Omega}_n})\bigr)^{-1}(\tilde{f}_{|_{\Omega_n}})
\bigr]_{|_\Omega}
& =  \bigl({\rm I}+(-\Delta_D^{\overline{\Omega}})\bigr)^{-1}f \quad \mbox{and}  \\
\lim_{n \to \infty} 
\bigl[\bigl({\rm I}+(-\Delta_D^{\Omega_n})\bigr)^{-1}(\tilde{f}_{|_{\Omega_n}})
\bigr]_{|_\Omega}
& =  \bigl({\rm I}+(-\Delta_D^{\overline{\Omega}})\bigr)^{-1}f
\end{align*}
in ${\mathcal{H}}$ for all $f\in {\mathcal{H}}$.
\end{enumerate}
\end{proposition}

\begin{remark}
\label{rem:comparison}
Strictly speaking, Proposition \ref{prop:approxDelta} has been proved in \cite{AD08}
for scalar valued functions $f\in L^2(\Omega,{\mathds{R}})$, and only the first part 
of \ref{prop:approxDelta-1} (\cite[Proposition~3.2]{AD08}) and the second part of 
\ref{prop:approxDelta-2} (\cite[Proposition~3.5]{AD08}) can be found in that reference.
Using \cite[Proposition~2.3]{AD08} establishing monotonicity properties of the 
resolvents of the weak-Dirichlet Laplacian and the pseudo-Dirichlet Laplacian with 
respect to the inclusion of domains, the other two limits are immediate. 
\end{remark}

\subsection{The weak-Dirichlet Stokes operator}
\label{subsec:weak}

Since the spaces ${\mathscr{V}}$ and ${\mathscr{X}}$ are 
dense subspaces of the Hilbert space ${\mathscr{H}}$, one can define 
two Dirichlet types of Stokes operators in ${\mathscr{H}}$.
Recall the form ${\frak{a}} : {\mathcal{W}} \times {\mathcal{W}} \to {\mathds{R}}$ 
from \eqref{eq:form-a}
\[
{\frak{a}}(u,v)
:=\langle \nabla u,\nabla v\rangle_{\mathcal{H}}
= \sum_{j=1}^d \langle \partial_ju,\partial_jv\rangle_{\mathcal{H}}.  
\]
Then ${\frak{a}}|_{{\mathscr{X}} \times {\mathscr{X}}}$ is a positive symmetric 
densely defined closed form in ${\mathscr{H}}$. Let ${\mathcal{B}}$ be the 
operator associated with ${\frak{a}}|_{{\mathscr{X}} \times {\mathscr{X}}}$.
Then ${\mathcal{B}}$ is self-adjoint and ${\mathcal{B}}$ is the
Stokes operator considered in \cite{Mo06}.

Since ${\mathscr{V}}\subset{\mathscr{X}}$ we can also define ${\mathcal{B}}_0$ 
to be the self-adjoint operator in ${\mathscr{H}}$ associated with the form 
${\frak{a}}|_{{\mathscr{V}} \times {\mathscr{V}}}$.
We call ${\mathcal{B}}_0$ the {\tt weak-Dirichlet Stokes operator}.
This Stokes operator is the one which was considered by H.\,Sohr in 
\cite[Chapter~3, \S2.1]{So01}.

The operators ${\mathcal{B}}$ and ${\mathcal{B}}_0$ are both negative
generators of analytic semigroups in ${\mathscr{H}}$.

Each of the cases above models differently spaces of divergence free vector fields
with zero boundary conditions.
As already mentioned before, they coincide in the 
case of bounded Lipschitz domains and consequently then also the two
operators ${\mathcal{B}}$ and ${\mathcal{B}}_0$ coincide.

The relation between the weak-Dirichlet Laplacian and the weak-Dirichlet
Stokes operator is described in the following commutative diagram:
\[
\xymatrix{
{\mathscr{V}}\  \ar@/_1.5pc/[dd]_{{\mathcal{B}}_0} \ar@{^{(}->}[d]_d
\ar@{^{(}->}[r]^{J_0} & {\mathcal{V}}\ar@{^{(}->}[d]^d 
\ar@/_-1.5pc/[dd]^{(-\Delta_D^\Omega)\ }\\
{\mathscr{H}}\  \ar@{^{(}->}[d]_d \ar@{^{(}->}[r]^{J} & {\mathcal{H}} 
\ar@<5pt>[l]^{{\mathbb{P}}=J'} \ar@{^{(}->}[d]^d\\
{\mathscr{V}}' & {\mathcal{V}}'\ar[l]^{{\mathbb{P}}_1=J_0'}
}
\]
where $J_0$ is the restriction of $J$ to ${\mathscr{V}}$ and ${\mathbb{P}}_1$,
its adjoint operator, is the extension of the Leray projection ${\mathbb{P}}$
to ${\mathcal{V}}'$. What this says in particular is that 
${\mathcal{B}}_0={\mathbb{P}}_1(-\Delta_D^\Omega)J_0$.

\subsection{The pseudo-Dirichlet Stokes operator}
\label{subsec:pseudo}

If we now restrict the form ${\frak{a}}$ to ${\mathscr{W}}\times{\mathscr{W}}$
we obtain a positive symmetric densely defined closed form in ${\mathscr{E}}$. 
We then define ${\mathcal{A}}$ to be the self-adjoint operator in ${\mathscr{E}}$ 
associated with ${\frak{a}}|_{{\mathscr{W}}\times{\mathscr{W}}}$. We call
${\mathcal{A}}$ the {\tt pseudo-Dirichlet Stokes operator}. It is the negative 
generator of an analytic semigroup in ${\mathscr{E}}$.

As said before, in the 
case of a bounded domain $\Omega$ with Lipschitz boundary, the spaces
${\mathscr{X}}$, ${\mathscr{V}}$ and ${\mathscr{W}}$ coincide as well as
the spaces ${\mathscr{H}}$ and ${\mathscr{E}}$, then so do the
operators ${\mathcal{B}}$, ${\mathcal{B}}_0$ and ${\mathcal{A}}$.

The relation between the pseudo-Dirichlet Laplacian and the pseudo-Dirichlet
Stokes operator is described in the following commutative diagram:
\[
\xymatrix{
{\mathscr{W}}\  \ar@/_1.5pc/[dd]_{{\mathcal{A}}} \ar@{^{(}->}[d]_d
\ar@{^{(}->}[r]^{L_0} & {\mathcal{W}}\ar@{^{(}->}[d]^d 
\ar@/_-1.5pc/[dd]^{(-\Delta_D^{\overline{\Omega}})\ }\\
{\mathscr{E}}\  \ar@{^{(}->}[d]_d \ar@{^{(}->}[r]^{L} & {\mathcal{H}} 
\ar@<5pt>[l]^{{\mathbb{Q}}=L'} \ar@{^{(}->}[d]^d\\
{\mathscr{W}}' & {\mathcal{W}}'\ar[l]^{{\mathbb{Q}}_1=L_0'}
}
\]
where $L_0$ is the restriction of $L$ to ${\mathscr{W}}$ and ${\mathbb{Q}}_1$,
its adjoint operator, is the extension of the projection ${\mathbb{Q}}$ from 
\S\ref{subsec:anotherdec} to ${\mathcal{W}}'$. What this says in particular is that 
${\mathcal{A}}={\mathbb{Q}}_1(-\Delta_D^{\overline{\Omega}})L_0$.

\section{Domain perturbation}
\label{sec:domainpert}

Similar results as those stated in Proposition~\ref{prop:approxDelta} hold for 
the different Dirichlet Stokes operators described above.
Roughly speaking, resolvents of the different Stokes operators converge to the 
resolvent of the weak-Dirichlet Stokes operator in ${\mathscr{H}}$ in the case 
of an increasing sequence of open sets and to the resolvent of the pseudo-Dirichlet 
Stokes operator in ${\mathscr{E}}$ in the case of a decreasing sequence of open sets.
 
\subsection{Increasing sequence of domains}
 
\begin{theorem}
\label{thm:increasingOmega}
Let $\Omega,\Omega_1,\Omega_2,\ldots$ be bounded open subsets of 
${\mathds{R}}^d$.
Suppose that $\Omega_n\uparrow\Omega$ as $n\to\infty$.
For all $n \in {\mathds{N}}$ denote by ${\mathbb{P}}_n$ the Leray projection from 
$L^2(\Omega_n,{\mathds{R}}^d)$ onto ${\mathscr{H}}_n$ (the corresponding space 
of divergence-free vector fields as in \eqref{def:decompH}), $I_n$ the identity 
operator on ${\mathscr{H}}_n$, 
${\mathscr{V}}_n$ the corresponding form domain
and ${\mathcal{B}}_0^{(n)}$ the corresponding weak-Dirichlet Stokes operator.
Then
\[
\lim_{n \to \infty}
\Bigl(\bigl(I_n+{\mathcal{B}}_0^{(n)}\bigr)^{-1}{\mathbb{P}}_n(f_{|_{\Omega_n}})
\Bigr)_{|_\Omega}^{\widetilde{\ \hphantom{|_\Omega}}}
= (I+{\mathcal{B}}_0)^{-1}f
\]
in $L^2(\Omega,{\mathds{R}}^d)$ for all $f\in{\mathscr{H}}$.
\end{theorem} 

\begin{proof}
For all $n \in {\mathds{N}}$ define 
$u_n = \bigl({\rm I}_n+{\mathcal{B}}_0^{(n)}\bigr)^{-1}
{\mathbb{P}}_n(f_{|_{\Omega_n}})$.
Then $u_n \in {\mathscr{V}}_n$ and ${\tilde{u}_n}{}_{|_{\Omega}} \in {\mathscr{V}}$ and 
\begin{equation}
\label{eq:u_n}
\int_{\Omega_n}\nabla u_n\cdot\nabla v + \int_{\Omega_n}u_n\cdot v
=\int_{\Omega_n} ({\mathbb{P}}_n(f_{|_{\Omega_n}})) \cdot v
=\int_{\Omega_n}f\cdot v
\end{equation}
for all $v \in {\mathscr{V}}_n$.
Choosing $v = u_n$ gives
\[
\int_{\Omega_n}|\nabla u_n|^2+\int_{\Omega_n}|u_n|^2
=\int_{\Omega_n}{\mathbb{P}}_n(f_{|_{\Omega_n}})\cdot u_n
\le \|f\|_2\Bigl(\int_{\Omega_n}|u_n|^2\Bigr)^{1/2}.  
\]
This implies that $({\tilde{u}_n}{}_{|_{\Omega}})_{n\in{\mathds{N}}}$ is a bounded 
sequence in ${\mathscr{V}}$.
Passing to a subsequence if necessary, there exists a $u \in {\mathscr{V}}$ 
such that $\displaystyle{\lim_{n\to \infty} {\tilde{u}_n}{}_{|_{\Omega}} = u}$ weakly 
in ${\mathscr{V}}$.
Let $v\in{\mathscr{D}}$. There exists an $N\in{\mathds{N}}$ such that 
$\tilde{v}_{|_{\Omega_n}}\in {\mathscr{D}}_n$ for all $n\ge N$.
By definition of $u_n$ we then have for all $n\ge N$ that
\[
{\frak{a}}({\tilde{u}_n}{}_{|_{\Omega}},v) 
+ \langle {\tilde{u}_n}{}_{|_{\Omega}},v \rangle_{L^2(\Omega,{\mathds{R}}^d)}
=\langle f,v\rangle_{L^2(\Omega,{\mathds{R}}^d)}.  
\]
Taking the limit as $n\to\infty$ we obtain that 
\begin{equation}
\label{eq:u,v}
{\frak{a}}(u,v)+\langle u,v\rangle_{L^2(\Omega,{\mathds{R}}^d)}
=\langle f,v\rangle_{L^2(\Omega,{\mathds{R}}^d)}.
\end{equation}
This is true for all $v\in {\mathscr{D}}$.
Then by continuity and density, \eqref{eq:u,v} is valid for all $v\in{\mathscr{V}}$.
This shows that $u\in {\sf{D}}({\mathcal{B}}_0)$ and $u=(I+{\mathcal{B}}_0)^{-1} f$.

It remains to show that
$\displaystyle{\lim_{n \to \infty} {\tilde{u}_n}{}_{|_{\Omega}} = u}$ strongly in 
$L^2(\Omega,{\mathds{R}}^d)$.
Since $\displaystyle{\lim_{n \to \infty} {\tilde{u}_n}{}_{|_{\Omega}} = u}$ weakly in 
$L^2(\Omega,{\mathds{R}}^d)$, 
\[
\liminf_{n \to \infty} 
\|{\tilde{u}_n}{}_{|_{\Omega}}\|_{L^2(\Omega,{\mathds{R}}^d)} 
\geq \|u\|_{L^2(\Omega,{\mathds{R}}^d)}.
\]
Comparing 
$\displaystyle{\limsup_{n\to\infty}\|u_n\|_2}$
and $\displaystyle{\liminf_{n\to\infty}\|u_n\|_2}$, it suffices to show
that 
\[
\limsup_{n \to \infty} 
\|{\tilde{u}_n}{}_{|_{\Omega}}\|_{L^2(\Omega,{\mathds{R}}^d)} 
\leq \|u\|_{L^2(\Omega,{\mathds{R}}^d)}.
\]
Let $n \in {\mathds{N}}$. Choose $v = u_n$ in \eqref{eq:u_n}.
Then 
\[
\int_\Omega |{\tilde{u}_n}{}_{|_{\Omega}}|^2
= \int_\Omega f\cdot {\tilde{u}_n}{}_{|_{\Omega}} 
- \int_{\Omega}\bigl|\nabla {\tilde{u}_n}{}_{|_{\Omega}}\bigr|^2.  
\]
Since $\displaystyle{\lim_{n \to \infty} {\tilde{u}_n}{}_{|_{\Omega}} = u}$ weakly in 
${\mathscr{V}}$, and hence in $H^1_0(\Omega,{\mathds{R}}^d)$, one deduces that 
\[
\lim_{n \to \infty} \int_\Omega f\cdot {\tilde{u}_n}{}_{|_{\Omega}} 
= \int_\Omega f\cdot u\quad\mbox{and} \quad 
\lim_{n \to \infty} \partial_k {\tilde{u}_n}{}_{|_{\Omega}} = \partial_k u \mbox{ weakly in }
L^2(\Omega,{\mathds{R}}^d)\mbox{ for all }k \in \{ 1,\ldots,d \}.
\]
This implies $\displaystyle{\|\partial_k u\|_{L^2(\Omega,{\mathds{R}}^d)} 
\leq \liminf_{n \to \infty} \|\partial_k {\tilde{u}_n}{}_{|_{\Omega}}\|_{L^2(\Omega,{\mathds{R}}^d)}}$.
Consequently, 
\begin{align*}
\limsup_{n\to\infty}\|{\tilde{u}_n}{}_{|_{\Omega}}\|_{L^2(\Omega,{\mathds{R}}^d)}^2
& =  \lim_{n\to\infty} \Bigl(\int_{\Omega}f\cdot {\tilde{u}_n}{}_{|_{\Omega}}\Bigr)
  -\liminf_{n\to\infty}\|\nabla {\tilde{u}_n}{}_{|_{\Omega}}\|_{L^2(\Omega,{\mathds{R}}^d)}^2\\
& \leq  \langle f, u \rangle_{L^2(\Omega,{\mathds{R}}^d)} 
- \|\nabla u\|_{L^2(\Omega,{\mathds{R}}^d)}^2 
=\|u\|_{L^2(\Omega,{\mathds{R}}^d)}^2,
\end{align*}
where the last equality follows from \eqref{eq:u,v}.
Then $\displaystyle{\lim_{n \to \infty} {\tilde{u}_n}{}_{|_\Omega} = u}$ strongly in 
$L^2(\Omega,{\mathds{R}}^d)$. One concludes by the fact that every sequence
for which every subsequence posseses a convergent subsequence to a unique 
limit is convergent.
\end{proof}

\subsection{Decreasing sequence of domains}

If $f \in L^2(\Omega,{\mathds{R}}^d)$, then we denote by 
$\tilde f \in L^2({\mathds{R}}^d,{\mathds{R}}^d)$ 
the extension by $0$ of $f$ to ${\mathds{R}}^d$.

\begin{theorem}
\label{thm:decreasingOmega}
Let $\Omega,\Omega_1,\Omega_2,\ldots$ be bounded open subsets of 
${\mathds{R}}^d$. Suppose that the $d$-dimensional Hausdorff measure of
$\partial \Omega$ and $\partial \Omega_n$ for all $n \in {\mathds{N}}$ is zero. 
Suppose that $\Omega_n \downarrow \Omega$ as 
$n\to\infty$. For all $n \in {\mathds{N}}$ denote by ${\mathbb{Q}}_n$ the 
projection from $L^2(\Omega_n,{\mathds{R}}^d)$ onto ${\mathscr{E}}_n$ 
(the corresponding space of divergence-free vector fields as in 
\eqref{def:anotherdecompH}), ${\mathcal{I}}_n$ the identity 
operator on ${\mathscr{E}}_n$, ${\mathscr{W}}_n$ the corresponding form domain
and ${\mathcal{A}}_n$ the corresponding pseudo-Dirichlet Stokes operator.
Then
\[
\lim_{n \to \infty} \Bigl(\bigl({\mathcal{I}}_n+{\mathcal{A}}_n\bigr)^{-1} 
\tilde f_{|_{\Omega_n}} \Bigr)_{|_{\Omega}}
= ({\mathcal{I}}+{\mathcal{A}})^{-1} f
\]
in $L^2(\Omega,{\mathds{R}}^d)$ for all $f\in{\mathscr{E}}$.
\end{theorem} 

\begin{proof}
First note that 
$\tilde{f}_{|_{\Omega_n}} \in {\mathscr{E}}_n$ for all $f \in {\mathscr{E}}$ and
$\tilde{u}_{|_{\Omega_n}} \in {\mathscr{W}}_n$ for all $u \in {\mathscr{W}}$ 
and all $n \in {\mathds{N}}$. Fix $f \in {\mathscr{E}}$.
For all $n \in {\mathds{N}}$ define 
$u_n:=\bigl({\mathcal{I}}_n+{\mathcal{A}}_n\bigr)^{-1}
(\tilde{f}_{|_{\Omega_n}})$.
Then $u_n \in {\mathscr{W}}_n$ and 
\begin{equation}
\label{eq:u_nV2}
\int_{{\mathds{R}}^d} \nabla\tilde{u}_n\cdot\nabla\tilde{v} 
+\int_{{\mathds{R}}^d} \tilde{u}_n\cdot\tilde{v}
= \int_{\Omega_n} \tilde{f}_{|_{\Omega_n}} \cdot \tilde{v}
= \int_{{\mathds{R}}^d} \tilde{f}\cdot \tilde{v},
\end{equation}
for all $v \in {\mathscr{W}}_n$.
Choosing $v = u_n$ gives
\[
\|\nabla \tilde{u}_n\|_{L^2({\mathds{R}}^d,{\mathds{R}}^d)}^2 
+ \|\tilde{u}_n\|_{L^2({\mathds{R}}^d,{\mathds{R}}^d)}^2
\leq \|\tilde f\|_{L^2({\mathds{R}}^d,{\mathds{R}}^d)} 
\|\tilde{u}_n\|_{L^2({\mathds{R}}^d,{\mathds{R}}^d)}.  
\]
Hence $(\tilde{u}_n)_{n\in{\mathds{N}}}$ is a bounded sequence in 
$H^1({\mathds{R}}^d,{\mathds{R}}^d)$.
Passing to a subsequence if necessary, there exists a 
$U \in H^1({\mathds{R}}^d,{\mathds{R}}^d)$ 
such that $\displaystyle{\lim_{n \to \infty} \tilde u_n = U}$ 
weakly in $H^1({\mathds{R}}^d,{\mathds{R}}^d)$.

We next show that $U = 0$ a.e.\ on $\overline{\Omega}^{\rm c}$.
Let $\Phi \in C_c^\infty({\mathds{R}}^d,{\mathds{R}}^d)$ and suppose that 
${\rm supp}\, \Phi \subset \overline \Omega^{\rm c}$.
If $n \in {\mathds{N}}$, then 
\[
\Bigl|\int_{{\mathds{R}}^d}\tilde{u}_n\cdot \Phi\Bigr|
\leq \|\tilde u_n\|_{L^2({\mathds{R}}^d,{\mathds{R}}^d)} 
\|\Phi\|_{L^2(\Omega_n,{\mathds{R}}^d)}
\leq \|\tilde f\|_{L^2({\mathds{R}}^d,{\mathds{R}}^d)} 
\|\Phi\|_{L^\infty({\mathds{R}}^d,{\mathds{R}}^d)} 
|\Omega_n \setminus \overline \Omega|^{1/2}.
\]
Since $\displaystyle{\lim_{n \to \infty} |\Omega_n \setminus \overline \Omega| = 0}$
it follows that 
\[
\int_{{\mathds{R}}^d}U\cdot \Phi
= \lim_{n\to\infty} \int_{{\mathds{R}}^d}\tilde{u}_n\cdot \Phi
= 0.
\]
So $U = 0$ a.e.\ on $\overline{\Omega}^{\rm c}$.
Set $u = U|_\Omega$. Then $u \in {\mathcal{W}}$.

To prove that $u\in {\mathscr{W}}$, it remains to prove that ${\rm div}\,U=0$
in ${\mathds{R}}^d$. This is straightforward since for all 
$\nabla p\in L^2({\mathds{R}}^d,{\mathds{R}}^d)$ and for all $n\in{\mathds{N}}$,
\[
\int_{{\mathds{R}}^d}U\cdot \nabla p
\xleftarrow[\infty\leftarrow n]{}\int_{{\mathds{R}}^d}\tilde{u}_n\cdot\nabla p=0.
\]
Now, taking the limit as $n$ goes to $\infty$ in \eqref{eq:u_nV2}
for $v\in{\mathscr{W}}$, we obtain that 
\begin{equation}
\label{eq:limU}
\int_{{\mathds{R}}^d}\nabla U\cdot\nabla\tilde{v} +\int_{{\mathds{R}}^d} U\cdot\tilde{v}
={\frak{a}}(u,v)+\langle u,v\rangle=\langle f,v\rangle
=\int_{{\mathds{R}}^d} \tilde{f}\cdot\tilde{v}.
\end{equation}
Therefore, $u\in{\sf D}({\mathcal{A}})$.
It remains to prove that 
${u_n}_{|_{\Omega}}\xrightarrow[n\to\infty]{}u$ strongly in 
$L^2(\Omega,{\mathds{R}}^d)$.
The proof is similar to the proof of Theorem~\ref{thm:increasingOmega}, 
comparing $\displaystyle{\liminf_{n\to\infty}\bigl\|{u_n}_{|_\Omega}\bigr\|}$
and $\displaystyle{\limsup_{n\to\infty}\bigl\|{u_n}_{|_\Omega}\bigr\|}$.
By weak convergence of $\bigl({u_n}_{|_{\Omega}}\bigr)_{n\in{\mathds{N}}}$ 
to $u$ in $L^2(\Omega,{\mathds{R}}^d)$, the inequality 
$\displaystyle{\liminf_{n\to\infty}\bigl\|{u_n}_{|_{\Omega}}\bigr\|_2\ge \|u\|_2}$
holds.
The proof of 
$\displaystyle{\limsup_{n\to\infty}\bigl\|{u_n}_{|_{\Omega}}\bigr\|_2\le \|u\|_2}$,
uses \eqref{eq:u_nV2} with $v=u_n$ and the fact that 
$\bigl(\tilde{u}_n\bigr)_{n\in{\mathds{N}}}$ converges weakly to $U$ in
$L^2({\mathds{R}}^d,{\mathds{R}}^d)$, so that
\begin{align*}
\limsup_{n\to\infty}\|{u_n}_{|_{\Omega}}\bigr\|_2^2
&\le \limsup_{n\to\infty}\|\tilde{u}_n\|_2^2
=\lim_{n\to\infty}\int_{{\mathds{R}}^d}\tilde{u}_n\cdot\tilde{f}
 -\liminf_{n\to\infty}\bigl\|\nabla \tilde{u}_n\bigr\|_2^2\\
&\le
\int_{{\mathds{R}}^d}U\cdot \tilde{f}-\|\nabla U\|_2^2
=\langle u,f\rangle-\|\nabla u\|_2^2=\|u\|_2^2
\end{align*}
by \eqref{eq:limU} with $v=u$.
\end{proof} 

\subsection{Comments}

The reader may want to compare Theorem~\ref{thm:decreasingOmega} 
and Theorem~\ref{thm:increasingOmega} with Proposition~\ref{prop:approxDelta}
and ask whether one can approximate the pseudo-Dirichlet Stokes operator in $\Omega$ 
with weak-Dirichlet Stokes operators in $\Omega_n$ where 
$\Omega_n\downarrow\Omega$ and the weak-Dirichlet Stokes operator in $\Omega$
with pseudo-Dirichlet Stokes operators in $\Omega_n$ where 
$\Omega_n\uparrow\Omega$. This is obviously true if the approximation domains
$\Omega_n$ are smooth and bounded since in this case, the weak-Dirichlet Stokes 
operator and the pseudo-Dirichlet Stokes operator coincide. In the case of increasing or
decreasing sequences of arbitrary domains, the strategy followed by \cite{AD08} 
(comparison of resolvents with respect to the inclusion of domains as in 
Remark~\ref{rem:comparison}) doesn't work: the Stokes problem is purely
vector-valued and the spaces involved are not Banach lattices.

{\small
\bibliographystyle{amsplain}

}
\end{document}